\documentclass[12pt]{amsart}

\usepackage{amsfonts,amsmath,amsthm, cite}
\usepackage{latexsym}

\newtheorem{theorem}{Theorem}[section]
\newtheorem{corollary}[theorem]{Corollary}

\theoremstyle{definition}

\newcommand{\TT}{\ensuremath{\mathbb{T}}}

\newcommand{\ZZ}{\ensuremath{\mathbb{Z}}}
\newcommand{\RR}{\ensuremath{\mathbb{R}}}

\newcommand{\CC}{\ensuremath{\mathbb{C}}}

\newcommand{\cg}{\ensuremath{\mathcal{G}}}

\newcommand{\vb}{\ensuremath{\mathbf{b}}}

\newcommand{\vf}{\ensuremath{\mathbf{f}}}
\newcommand{\vu}{\ensuremath{\mathbf{u}}}

\newcommand{\vw}{\ensuremath{\mathbf{w}}}
\newcommand{\vv}{\ensuremath{\mathbf{v}}}
\newcommand{\vz}{\ensuremath{\mathbf{z}}}
\newcommand{\one}{\ensuremath{\mathbf{1}}}

\def \< {\langle}
\def \> {\rangle}

\begin{document}

\title[An improvement on the Delsarte-type LP-bound]{An improvement on the Delsarte-type LP-bound with application to MUBs}

\author[M. Matolcsi]{M. Matolcsi}
\address{M. M.: Alfr\'ed R\'enyi Institute of Mathematics,
Hungarian Academy of Sciences POB 127 H-1364 Budapest, Hungary
Tel: (+361) 483-8307, Fax: (+361) 483-8333}
\email{matolcsi.mate@renyi.mta.hu}

\author[M. Weiner]{M. Weiner}
\address{M. W.: Budapest University of Technology \& Economics (BME),
Department of Mathematical Analysis,
M\"uegyetem rkp. 3--9,  H-1111 Budapest, Hungary}
\email{mweiner@math.bme.hu}

\thanks{M. M. was supported bt OTKA grant No.\! 109789 and by ERC-AdG 321104.
M. W. was supported by OTKA grant No.\! 104206 and by the ``Bolyai J\'anos" Research Scholarship of the Hungarian Academy of Sciences''.}

\begin{abstract}
The linear programming (LP) bound of Delsarte can be applied to several problems in various branches of mathematics. We describe a general Fourier analytic method to get a slight improvement on this bound. We then apply our method to the problem of mutually unbiased bases (MUBs) to prove that the Fourier family $F(a,b)$ in dimension 6 cannot be extended to a full system of MUBs.
\end{abstract}

\maketitle

\bigskip

\section{Introduction}

The linear programming bound of Delsarte was first applied in \cite{delsarte} in coding theory to the following problem: determine the maximal cardinality $A(n,d)$ of binary codewords of length $n$ such that each two of them differ in at least $d$ coordinates. In the past decades the method of Delsarte has been applied to several other problems, most notably to sphere packings \cite{cohnelkies}, and the unit-distance graph of $\RR^n$ \cite{vallfilho}. In this note we will recall a Fourier analytic formulation of Delsarte's bound \cite{mubfourier}. This is not the most general form of the method but it captures most of the applications and is simple enough to require only elementary Fourier analysis.

\medskip

After the description of the LP-bound we give a general method to get a slight improvement on it. Unfortunately, this improvement is usually very small numerically. However, in ceratin problems the Delsarte bound is already sharp in itself, and {\it any} improvement on it can lead to non-existence results. This is exactly the situation in the problem of mutually unbiased bases (MUBs), as described in \cite{mubfourier}. We will apply our improved bound to show that the Fourier family $F(a,b)$ of complex Hadamard matrices cannot be extended to a full system of MUBs in dimension 6. This result was previously proven by a massive computer search after a discretization scheme \cite{pauli}. Our proof here is completely elementary and could also lead to similar results for other families of complex Hadamard matrices.

\section{The Delsarte bound}

We recall the Fourier analytic formulation of Delsarte's bound, as described in \cite{mubfourier} and \cite{intersective1}.

\medskip

Let $G$ be a compact Abelian group, and let a symmetric subset $A=-A\subset G$, $0\in A$ be given. We will call $A$ the 'forbidden' set. We would like  to determine the maximal cardinality of a set $B=\{b_1, \dots b_m\}\subset G$ such that all differences $b_j-b_k\in A^c\cup \{0\}$ (in other words, all differences avoid the forbidden set $A$).

\medskip

We will also need the dual group $\hat{G}$, i.e the group of multiplicative characters from $G$ to $\CC$. In this section we will use the multiplicative notation for the operation of the dual group, i.e. for $\gamma_1, \gamma_2\in\hat{G}$ and $x\in G$ we define $(\gamma_1\gamma_2)(x)=\gamma_1(x)\gamma_2(x)$. In particular, the unit element of the dual group (i.e. the constant 1 function) will be denoted by $\one\in \hat{G}$.

\medskip

We will use the normalized Haar measure on $G$ (i.e. the measure of $G$ is 1), and the following definition for the Fourier transform for any function $f: G\to \CC$: $\hat{f}(\gamma)=\int_{x\in G} f(x)\gamma(x)dx$.

\medskip

Let us now recall Delsarte's bound in this formulation \cite{mubfourier, intersective1}. We also recall the proof here because we will need it later.

\begin{theorem}\label{delsartelemma} (Delsarte's bound)\\
Assume we have a witness function $h: G\to \RR$ with the following properties: $h(x)\le 0$ for all $x\in A^c$, $\hat{h}(\gamma)\ge 0$ for all $\gamma\in \hat{G}$, the Fourier inversion formula is valid for $h$ (in particular, $h$ can be any finite linear combination of characters). Then for any $B=\{b_1, \dots b_m\}\subset G$ such that $b_j-b_k\in A^c\cup \{0\}$ we have $|B|\le \frac{h(0)}{\hat{h}(\one)}$.
\end{theorem}

\begin{proof}
For any $\gamma\in \hat{G}$ define $\hat{B}(\gamma)=\sum_{j=1}^m \gamma (b_j)$, and let us evaluate
\begin{equation}\label{eq1}
S=\sum_{\gamma\in \hat{G}} |\hat{B}(\gamma)|^2 \hat{h}(\gamma).
\end{equation}

All terms are nonnegative, and the term corresponding to $\gamma=\one$ gives $|\hat{B}(\one)|^2\hat{h}(\one)=|B|^2\hat{h}(\one)$. Therefore

\begin{equation}\label{eq2}
S\ge |B|^2\hat{h}(\one).
\end{equation}

On the other hand, $|\hat{B}(\gamma)|^2= \sum_{j,k}\gamma(b_j-b_k)$, and therefore $S=\sum_{\gamma,j,k}\gamma(b_j-b_k)\hat{h}(\gamma)$. Summing up for fixed $j,k$ we get \\
$\sum_{\gamma}\gamma(b_j-b_k)\hat{h}(\gamma)=h(b_j-b_k)$ (the Fourier inversion formula for $h$), and therefore $S=\sum_{j,k} h(b_j-b_k)$. Notice that $j=k$ happens $|B|$-many times, and all the other terms (when $j\ne k$) are non-positive because $b_j-b_k\in A^c$, and $h$ is required to be non-positive there. Therefore
\begin{equation}\label{eq3}
S\le h(0)|B|.
\end{equation}
Comparing the two estimates \eqref{eq2}, \eqref{eq3} we obtain
\begin{equation}\label{delbound}
|B|\le \frac{h(0)}{\hat{h}(\one)}.
\end{equation}
\end{proof}

\medskip

In principle, the best witness function $h$ can be found by linear programming if $G$ is finite. In practice, the cardinality of $G$ needs to be small enough for the LP-code to be executed.

\section{Improving the Delsarte bound}

When obtaining the lower bound \eqref{eq2} we have thrown away all non-trivial terms ($\gamma\ne \one$) on the right hand side of \eqref{eq1}. This seems rather wasteful. We will try to make use of the remaining terms in this section.

\medskip

Assume we have some further restriction on the set $B$: not only must each $b_j-b_k$ fall into $A^c\cup \{0\}$ but also $B$ must be contained in some prescribed set $C\subset G$.

\begin{theorem}\label{Kfunctionbound}
Let $C\subset G$ be a measurable subset.
Assume $h$ is a witness function as in the Delsarte bound: $h: G\to \RR$, $h(x)\le 0$ for all $x\in A^c$, $\hat{h}(\gamma)\ge 0$ for all $\gamma\in \hat{G}$, and the Fourier inversion formula holds for $h$. Let $Null$ denote the set of $\gamma$'s where $\hat{h}(\gamma)=0$. Assume furthermore that we have another witness function $K:G\to \CC$ with the following properties: $K(x)\ge 1$ for $x\in C$, $\hat{K}(\one)=0$, and $\hat{K}(\gamma)=0$ for all $\gamma\in Null$. Then any $B\subset C$ such that $B-B\subset A^c \cup \{0\}$ satisfies
\begin{equation}\label{K}
|B|\le \frac{h(0)}{\hat{h}(\one)+\left (\sum_{\gamma\notin Null}\frac{|\hat{K}(\gamma)|^2}{\hat{h}(\gamma)}\right )^{-1}}
\end{equation}
\end{theorem}
\begin{proof}
We will make use of the non-trivial terms in \eqref{eq1}. Namely,
\begin{equation}\label{Kproof}
\left (\sum_{\gamma\ne \one, \gamma\notin Null}|\hat{B}(\gamma)|^2 \hat{h}(\gamma)\right )\left ( \sum_{\gamma\ne \one, \gamma\notin Null}\frac{|\hat{K}(\gamma)|^2}{\hat{h}(\gamma)} \right )\ge
\end{equation}
\begin{equation*}
\left |\sum_{\gamma\ne \one, \gamma\notin Null}\hat{B}(\gamma) \overline{\hat{K}(\gamma)}\right |^2 =
\left |\sum_{\gamma\in\hat{G}}\hat{B}(\gamma) \overline{\hat{K}(\gamma)}\right |^2=\left |\sum_{x\in G} B(x)\overline{K(x)}\right |^2=
\end{equation*}
\begin{equation*}
\left |\sum_{x\in C} B(x)\overline{K(x)}\right |^2\ge |B|^2
\end{equation*}
where we used Cauchy-Schwarz, the assumptions on $\hat{K}(\gamma)$, Parseval, and the assumptions on $B(x)$ and $K(x)$, respectively. Therefore, we get an improved version of \eqref{eq2}, namely:
\begin{equation}
S\ge |B|^2\hat{h}(\one)+ \frac{|B|^2}{\sum_{\gamma\ne 0, \gamma\notin Null}\frac{|\hat{K}(\gamma)|^2}{\hat{h}(\gamma)}}.
\end{equation}
Comparing this with \eqref{eq3} yields the desired bound \eqref{K}.
\end{proof}

We see that Theorem \ref{Kfunctionbound} requires a combination of two witness functions $h(x)$ and $K(x)$ (as well as a prescribed set $C$ in which $B$ is assumed to be located). Unfortunately, it is not at all clear how to optimize $h$ and $K$ in actual applications. The best chance to apply \eqref{K} successfully arises in situations when the Delsarte bound \eqref{delbound} is already sharp. In such cases the sheer {\it existence of any} $K$ can lead to non-existence results, as we explain in the next paragraphs. Let us first state a corollary, which describes the usual situation in which Theorem \ref{Kfunctionbound} can be used.

\begin{corollary}\label{cor1}
Assume that for a given forbidden set $0\in A=-A\subset G$ we already have a witness function $h(x)$ as in Theorem \ref{delsartelemma}, testifying that $|B|\le  \frac{h(0)}{\hat{h}(\one)}=m\in \ZZ$ for any set $B\subset G$ such that $B-B\subset A^c\cup \{0\}$. Assume also that a few elements $b_1, \dots, b_k\in G$ are given with the property that $b_i-b_j\in A^c$ for all $i\ne j$. Let $D$ denote the set of elements in $G$ (different from $b_1, \dots, b_k$) such that $d-b_j\in A^c$ for all $j=1, \dots k$.  Assume furthermore that we have a second witness function $K(x)$ such that $\hat{K}(\one)=0$, $\hat{K}(\gamma)=0$ for all $\gamma\in Null$, and $\sum_{j=1}^{k} K(b_j)=1$ while $K(x)> \frac{-1}{m-k}$ for all $x\in D$ or $K(x)< \frac{-1}{m-k}$ for all $x\in D$. Then, for any $B\subset G$ such that $b_1, \dots, b_k\in B$ and $B-B\subset A^c\cup\{0\}$ we have that $|B|\le m-1$.
\end{corollary}
\begin{proof}
This is a direct consequence of the proof of Theorem \ref{Kfunctionbound}. Assume by contradiction that $|B|=m$. By the penultimate term of inequality \eqref{Kproof} this can only happen if $\sum_{x\in C} B(x)\overline{K(x)}= 0$ (otherwise, using \eqref{Kproof}, we could get {\it some} improvement on the bound $|B|\le \frac{h(0)}{\hat{h}(\one)}=m$). However, by the conditions above we have $\sum_{x\in C} B(x)\overline{K(x)}=1+\sum_{b\in B, b\in D} K(b)\ne 0$, because the sum is larger than zero if $K(x)> \frac{-1}{m-k}$ for all $x\in D$, while it is smaller than zero if $K(x)< \frac{-1}{m-k}$ for all $x\in D$. Therefore, $B$ can contain at most $m-1$ elements.
\end{proof}

\section{Application to mutually unbiased bases (MUBs)}

We now turn to an elegant application of Corollary \ref{cor1} to the problem of mutually unbiased bases (MUBs). What makes this application possible is the fact that the Delsarte bound is already sharp in the MUB problem, as explained below (see also \cite{mubfourier} for more details, where this idea was introduced).

\medskip

We will use the formulation of the MUB problem in terms of complex Hadamard matrices.
A complex Hadamard matrix $H$ is a complex orthogonal matrix whose entries are of modulus 1. Two such matrices $H_1, H_2$ are called unbiased if any two columns $\vu\in H_1, \vw\in H_2$ satisfy $|\langle \vu, \vw \rangle = |\sqrt{n}$. A convenient formulation of the MUB problem is whether there exists a system $H_1, \dots , H_n$ of pairwise mutually unbiased complex Hadamard matrices (MUHs) in dimension $n$. The answer is known to be positive if $n$ is a prime-power (se e.g. \cite{BBRV,Iv,KR,WF}), while the problem is open for any non-prime-power dimensions.

\medskip

Assume that $H_1, \dots H_r$ is a system of mutually unbiased complex Hadamard matrices.
The columns of each $H_j$ are unimodular vectors which can be considered as elements of
the group $G=\TT^n$, where $\TT=\{ z\in \CC: |z|=1\}$. The group operation on $G$ is
coordinate-wise multiplication, the unit element is the constant 1 vector denoted by
$\one$, and the dual group $\hat G$ is $\ZZ^n$ the unit of which is denoted by $\bf{0}$. (In this particular application it is more convenient to use the multiplication operation on $\cg$ and the addition operation on $\hat \cg$.)
Also, any two distinct column vectors in this system must be either orthogonal or
unbiased to each other (depending on whether they belong to the same matrix or not).
Therefore, any two distinct columns $\vv, \vw$ satisfy $|\sum_{j=1}^{n} v_j
\overline{w_j}|^2= |\sum_{j=1}^{n} v_j /w_j|^2=0 \ \textrm{or} \ n$. In the language of Theorem \ref{delsartelemma} this means that $\vv/\vw$ must fall into the set $A^c=\{\vz\in\TT: \sum_{j=1}^{n} z_j=0\}\cup \{\vz\in\TT: |\sum_{j=1}^{n} z_j|^2=n\}$. Consider now the
witness function $h: G\to \RR$, $h(\vz)=|z_1+\dots +z_n|^2 (|z_1+\dots z_n|^2-n)$. It is
fairly easy to check that this function satisfies all the conditions listed in Theorem
\ref{delsartelemma}, and $\frac{h(\one)}{\hat{h}(\bf{0})}=n^2$. This testifies that the
total number of column vectors in the matrices $H_j$ cannot be larger than $n^2$, and
hence the number of MUHs cannot be larger than $n$.

\medskip

Of course, this bound cannot be improved if $n$ is a prime-power,
because a full system of MUBs (or, equivalently, MUHs) actually exists
for such $n$. However, for any given Hadamard matrix $H$ one can try to use
Corollary \ref{cor1} to rule out the possibility that $H$ could
be part of a full system of MUHs. In view of the witness function $h$ above, all we
need is a suitable function $\vz\mapsto K(\vz)$ which is a linear
combination of terms of the form $z_iz_j\overline{z}_k\overline{z}_l$
with $\{i,j\}\neq \{k,l\}$ (since these
are exactly the non-constant terms appearing in $h$, this will ensure $\hat K(\gamma)=0$ for all $\gamma\in Null$) and satisfies the bounds given in
Corollary \ref{cor1}.

\medskip

We shall now demonstrate the power of this method by an actual example.
In dimension $6$, the fact that no Hadamard matrix
$F(a,b)$ of the Fourier family can be part of a full system of MUHs was proven by a massive computer search
using a discretization scheme in \cite{pauli}. In particular, there is no way to check that proof by hand. Here we shall give a
simple proof of this statement requiring no computer assistance.

\begin{theorem}\label{nofab}
In dimension $n=6$, no complex Hadamard matrix
$F(a,b)$ of the Fourier family, or $F^T(a,b)$ of the transposed Fourier family can be extended to a full system of MUHs.
\end{theorem}

\begin{proof}
It is trivial that any complex Hadamard matrix $H$ can be extended to a full system of MUHs
if and only if its conjugate $\overline{H}$, adjoint $H^\ast$ or transpose $H^T$ can. Therefore we may restrict our attention to the transposed Fourier family $F^T(a,b)$. The usual parametrization of $F^T(a,b)$ is given in \cite{karol}. However, it will be more convenient for us to permute rows and columns and work with an equivalent parametrization given by the following column vectors:

\begin{equation}
\label{eq:fab}
\left(\begin{matrix}
\vf_0 \\ \vf_0
\end{matrix}\right),
\,
\left(\begin{matrix}
\vf_0 \\ -\vf_0
\end{matrix}\right),
\,
\left(\begin{matrix}
\vf_1 \\ a\vf_1
\end{matrix}\right),
\,
\left(\begin{matrix}
\vf_1 \\ -a\vf_1
\end{matrix}\right),
\,
\left(\begin{matrix}
\vf_2 \\ b \vf_2
\end{matrix}\right),
\,
\left(\begin{matrix}
\vf_2 \\ -b\vf_2
\end{matrix}\right)
\end{equation}
where
\begin{equation}
\left(\begin{matrix}
\vf_0 &
\vf_1 &
\vf_2
\end{matrix}\right) =
\left(\begin{matrix}
1 & 1 & 1 \\
1 & e^{i\frac{2\pi}{3}}  &  e^{-i\frac{2\pi}{3}} \\
1 & e^{-i\frac{2\pi}{3}} & e^{i\frac{2\pi}{3}}
\end{matrix}\right)
\end{equation}
and $a,b\in \TT$ are two complex unit parameters. With a slight abuse of notation we will still denote the matrix formed by the six columns above by $F^T(a,b)$. Also, we will denote the columns in \eqref{eq:fab}  by $\vb_1, \dots, \vb_6$, in accordance with the notation of Corollary \ref{cor1}. Note that the set $D$ appearing in Corollary \ref{cor1} consists of the vectors $\vz\in \TT^6$ which are unbiased to $\vb_1, \dots, \vb_6$.

\medskip

Now we define the second witness function $K(\vz).$ For any $\vz\in \TT^6$ written in the form
\begin{equation}
\vz=
\left(\begin{matrix}
\vz_{\uparrow} \\ \vz_{\downarrow}
\end{matrix}\right),\;\;\;\;\;
\vz_{\uparrow}, \vz_{\downarrow}\in \TT^3
\end{equation}
let $K(\vz)=$
\begin{equation}
\frac{1}{N} \left[
\left(\langle\vz_{\uparrow},\vf_0,\rangle \,
\langle \vf_0,\vz_{\downarrow}\rangle \right)^2
+
\left(\langle\vz_{\uparrow},\vf_1,\rangle \,
\langle a\vf_1,\vz_{\downarrow}\rangle \right)^2
+
\left(\langle\vz_{\uparrow},\vf_2,\rangle \,
\langle b\vf_2,\vz_{\downarrow}\rangle \right)^2
\right]
\end{equation}
where the normalizing term $N=6\cdot (3\cdot 3)^2=486$ is chosen so
that the sum taken over the columns $\vb_1,\ldots,\vb_6$ of $F^T(a,b)$ is
$\sum_{j=1}^6K(\vb_j)=1$. (This is trivial to check.)

\medskip

The function $K$ is a linear
combination of terms of the form $z_iz_j\overline{z}_k\overline{z}_l$
with $\{i,j\}\neq \{k,l\}$, just as the witness function $h$. In order to apply Corollary \ref{cor1} we need to estimate the value of
$K(\vz)$ whenever $\vz\in\TT^6$ is an unbiased
vector to our Hadamard matrix $F^T(a,b)$. We will show that $K(\vz)<-\frac{1}{30}$, as required in Corollary \ref{cor1}. It is interesting to note here that we will be able to do this {\it without} the explicit knowledge of the unbiased vectors $\vz$.

\medskip

In what follows, suppose $\vz= \left(\begin{matrix} \vz_{\uparrow} \\
\vz_{\downarrow} \end{matrix}\right)\in \TT^6$ is an unbiased vector to the
matrix $F^T(a,b)$. In particular, it is unbiased to the first
two columns $\vb_1, \vb_2$ listed at (\ref{eq:fab}), which means
\begin{equation}
|\langle \vz_\uparrow,\vf_0 \rangle +
\langle \vz_\downarrow,\vf_0 \rangle |^2
=6=
|\langle \vz_\uparrow,\vf_0 \rangle -
\langle \vz_\downarrow,\vf_0 \rangle |^2.
\end{equation}
This implies that the product
$\langle \vz_\uparrow,\vf_0 \rangle
\overline{\langle \vz_\downarrow,\vf_0 \rangle }$
(whose square appears in the definition of $K(\vz)$) is purely
imaginary. Thus
$|\langle \vz_\uparrow,\vf_0 \rangle|^2+|\langle
\vz_\downarrow,\vf_0 \rangle|^2=6$ and the term
$\left(\langle\vz_{\uparrow},\vf_0,\rangle \,
\langle \vf_0,\vz_{\downarrow}\rangle \right)^2 =
$
\begin{equation}
-
|\langle \vz_\uparrow,\vf_0 \rangle|^2 \,|\langle
\vz_\downarrow,\vf_0 \rangle|^2 =
-|\langle \vz_\uparrow,\vf_0 \rangle|^2 \left(6-
|\langle \vz_\uparrow,\vf_0 \rangle|^2\right).
\end{equation}
Similarly, introducing the notation $s_j=|\langle \vz_\uparrow,\vf_j \rangle|^2$ ($j=0,1,2$) we have that
\begin{equation}\label{sj}
6-s_j=|\langle \vz_\downarrow,\vf_j \rangle|^2
\end{equation}
and obtain
\begin{equation}
K(\vz)
=-\frac{1}{N}(s_0(6-s_0)+s_1(6-s_1)+s_2(6-s_2)).
\end{equation}
This can be further simplified using that
\begin{equation}\label{s9}
s_0+s_1+s_2=\sum_{j=0}^2 |\langle \vz_\uparrow,f_j\rangle|^2 = 3
\|\vz_\uparrow\|^2 =3\cdot 3 = 9
\end{equation}
as $
\frac{1}{\sqrt{3}}
\vf_0,
\frac{1}{\sqrt{3}}
\vf_1,
\frac{1}{\sqrt{3}}
\vf_2$ is an orthonormal basis of $\mathbb C^3$. Therefore, after simplification,
\begin{equation}
K(\vz)=
\frac{s_0^2+s_1^2+s_2^2-54}{486}
\; \textrm{where}\;\; s_0,s_1,s_2\ge 0, \; s_0+s_1+s_2=9.
\end{equation}
Note that in general the value of $K(\vz)$ is not necessarily real,
but the formula above shows that it is so when evaluated at a vector $\vz$ which is unbiased to $F^T(a,b)$.

\medskip

Note that $0\le s_0, s_1, s_2 \le 6$ by \eqref{sj}. Furthermore, we will see that the values of $s_j$ cannot be close to 0 or 6. Indeed, consider the following optimization problem: minimize $|\langle \vf_0, \vu \rangle |^2$ over all $\vu\in \TT^3$ subject to the constraints $|\langle \vf_1, \vu \rangle |^2\le 6$, $|\langle \vf_2, \vu \rangle |^2\le 6$. We can assume without loss of generality that the first coordinate of $\vu$ is 1, so that
\begin{equation}
\vu=\left(\begin{matrix}
1  \\
e^{i\alpha}  \\
e^{i\beta}
\end{matrix}\right)
\end{equation}
For the discussion below introduce the notations $g_j(\alpha, \beta)=|\langle \vf_j, \vu \rangle |^2$, $j=0, 1, 2$.
The two-parameter optimization problem above can be solved by standard methods. First, by a trivial compactness argument the minimum is actually attained at some point $(\alpha^\ast, \beta^\ast)$. Second, the point $(\alpha^\ast, \beta^\ast)$ must satisfy one of the following:

\smallskip

\noindent (i) the derivative of $g_0(\alpha, \beta)$ is zero at  $(\alpha^\ast, \beta^\ast)$;\\
(ii) both constraints hold with equality, i.e. $g_j (\alpha^\ast, \beta^\ast)=6$ for $j=1, 2$;  \\
(iii) one constraint holds with equality (say,  $g_1(\alpha^\ast, \beta^\ast)=6$), and by the method of Lagrangian multipliers we have
$(\partial_\alpha g_0) (\partial_\beta g_1)= (\partial_\beta g_0) (\partial_\alpha g_1$) at $(\alpha^\ast, \beta^\ast)$.

\medskip

It is easy to see that (ii) cannot happen (because $g_1(\alpha, \beta)+g_2(\alpha, \beta)\le 9$), while the cases of (i) are easy to determine and they either do not satisfy the side constraints $g_j(\alpha, \beta)\le 6$, or do not lead to the actual minimum of the optimization problem. The actual minimum occurs in case (iii), which leads to the following system of equations (after introducing the variables $x_1=\cos \alpha, x_2=\cos \beta, y_1=\sin \alpha, y_2=\sin \beta)$:
\begin{eqnarray*}\label{ab}
 - x_1 - x_2 - x_1 x_2 + \sqrt{3} y_1 - \sqrt{3} x_2 y_1 - \sqrt{3} y_2 +
  \sqrt{3} x_1 y_2 - y_1 y_2 - 3 = 0,\\
2 \sqrt{3} x_2 y_1 - 4 \sqrt{3} x_1 x_2 y_1 + 2 \sqrt{3} x_2^2 y_1 +
  2 \sqrt{3} x_1 y_2 +
   2 \sqrt{3} x_1^2 y_2 \\
   -4 \sqrt{3} x_1 x_2 y_2
   -2 \sqrt{3} y_1^2 y_2 - 2 \sqrt{3} y_1 y_2^2 = 0,\\
   x_1^2 + y_1^2 - 1 =0,\\
x_2^2 + y_2^2 - 1 = 0.
\end{eqnarray*}

This system can be solved exactly (by hand if necessary, but more conveniently with computer algebra), and leads to the lower bound $g_0(\alpha, \beta)\ge c=\frac{3}{2}-\frac{3}{2}\sqrt{16\sqrt{6}-39}$, showing that $s_0\ge c> 0.843$. Consequently, we also obtain $s_0\le 6-c$, because $6-s_0=|\langle \vz_\downarrow,\vf_0 \rangle|^2$ must satisfy the same optimization problem. The same argument applies to $s_1$ and $s_2$, giving the bounds $c\le s_0, s_1, s_2 \le 6-c$. Together with the fact that $s_0+s_1+s_2=9$ this implies $s_0^2+s_1^2+s_2^2\le c^2+(6-c)^2+3^2<37$. Hence $K(\vz)=\frac{s_0^2+s_1^2+s_2^2-54}{486}<-\frac{17}{486}<-\frac{1}{30}$, and Corollary \ref{cor1} applies.
\end{proof}

\medskip

We remark that Corollary \ref{cor1} could have further similar applications in the future. For example, it is natural to try to prove in a similar manner that in dimension 6 the matrices $D(c)$ of the Dita-family cannot be extended to a full system of MUHs. The method could also be applied to the Fourier matrix $F_n$ for any composite $n$, in which case we conjecture that $F_n$ cannot be extended to a full set of MUHs. More generally, in any problem where Delasrte's method gives an upper bound, Theorem \ref{Kfunctionbound} might lead to an improvement if a suitable second witness function $K$ can be found.

\medskip

Finally, we remark that Corollary \ref{cor1} could be applied together with the discretization scheme described in \cite{morapetis}. A witness function $K$ may exist even if the entries of the first Hadamard matrix $H_1$ are only known to some precision. In principle, this can lead to a major improvement of the running time of the discretization method.

\section{Acknowledgement}

The authors thank I. Z. Ruzsa for helpful discussions on the subject and for an alternative solution
of the optimization problem presented at the end of the proof of Theorem \ref{nofab}.

\end{document}